\documentclass[12pt, one side,emlines]{amsart}
\usepackage{amssymb,latexsym,xy,eucal,mathrsfs, graphicx, tikz}
\textwidth=15.1cm \textheight=24cm \theoremstyle{plain}
\newtheorem{lem}{Lemma}[section]

\newtheorem{theorem}[lem]{Theorem}
\newtheorem{thm}[lem]{Theorem}
\newtheorem{prop}[lem]{Proposition}
\newtheorem{cor}[lem]{Corollary}

\theoremstyle{definition}

\newtheorem{rem}[lem]{Remark}
\newtheorem{note}[lem]{Note}

\newtheorem{defn}[lem]{Definition}

\numberwithin{equation}{section}  \voffset
-55truept \hoffset -18truept
\begin{document}
\baselineskip 16truept 
\title{Natural Partial Order on rings with involution}
\author[Avinash Patil and B. N. Waphare]%
{Avinash Patil$^1$ and B. N. Waphare$^2$}
\address{$1$ Department of Mathematics, Garware
College of
Commerce, Pune-411004, India.}
\address{$2$ Center For Advanced Studies In Mathematics, Department of Mathematics, Savitribai Phule Pune University,
Pune-411007, India.}
 \email{avipmj@gmail.com}
 \email{waphare@yahoo.com;
 bnwaph@math.unipune.ac.in}
  \maketitle

\begin{abstract}
In this paper, we introduce a partial order on rings with involution, which is a
generalization of the  partial order on the set of projections in a
Rickart $*$-ring. We prove that, a $*$-ring with the natural partial order form a sectionally semi-complemented poset. It is proved that every interval $[0,x]$ forms an orthomodular lattice in case of abelian Rickart $*$-rings.  
The concepts of generalized comparability $(GC)$ and partial
comparability $(PC)$ are extended to involve all the elements of
a $*$-ring. Further, it is proved that these concepts are equivalent in finite abelian Rickart $*$-rings.
\end{abstract}

\noindent {\bf Keywords :} $*$-ring, partial order, generalized comparability, partial comaparbility.
{\bf MSC(2010):} {Primary: $16$W$10$, Secondary: $06$F$25$}
\section{introduction}  An \textit{involution} $*$ on an associative ring $R$ is a mapping such that $(a+b)^*=a^*+b^*$,
$(ab)^*=b^*a^*$ and $(a^*)^*=a$, for all $ a,b\in R$. A ring with
involution $*$ is called a {\it $*$-ring}. Clearly, identity mapping is an
involution if and only if the ring is commutative. An element $e$ of
a $*$-ring $R$ is a \textit{projection} if $e=e^2$ and $e=e^*$. For
a nonempty subset $B$ of $R$, we write $r(B)=\{x\in R\colon
bx=0,\forall b\in B\}$, and call a \textit{right annihilator} of $B$
in $R$. A \textit{Rickart} $*$-\textit{ring} $R$ is a  $*$-ring   in which
right annihilator of every element is generated, as a right ideal,
by a projection in $R$. Every Rickart $*$-ring contains unity. For
each element $a$ in a Rickart $*$-ring, there is unique projection
$e$ such that $ae=a$ and $ax=0$ if and only if $ex=0$, called a {\it
right projection} of $a$ denoted by $RP(a)$. In fact,
$r(\{a\})=(1-RP(a))R$. Similarly, the left annihilator $l(\{a\})$ and the left
projection $LP(a)$ are defined for each element $a$ in a Rickart $*$-ring $R$. The
set of projections $P(R)$ in a Rickart $*$-ring $R$ forms a lattice, denoted by
$L(P(R))$, under the partial order `$e\leq f$ if and only if
$e=fe=ef$'. In fact, $e\vee f=f+RP(e(1-f))$ and $e\wedge
f=e-LP(e(1-f))$. This
lattice is extensively studied by I. Kaplanski \cite{kp}, S. K.
Berberian \cite{skb}, S. Maeda in \cite{ms2, ms3} and others.

    Drazin \cite{d}, was the first to introduce ``$*$-order'' to involve all elements, where $*$-order is given by, $a\underset{*}{\leqslant} b$ if and only if $a^*a=a^*b$ and $aa^*=ba^*$, which is a partial order on a semigroup with
 proper involution({\it i.e.}, $aa^*=0$ implies $a=0$). In
 particular, with the obvious choices for $*$-rings with proper involution, all commutative rings
 with no nonzero nilpotent elements, all Boolean rings, the ring
 $B(H)$ of all bounded linear operators on any complex Hilbert space
 $H$, the Rickart $*$-ring. Janowitz \cite{j} studies $*$-order on Rickart $*$-ring. Thakare and Nimbhorkar \cite{nsk} used $*$-order on Rickart $*$-ring and generalized the comparability axioms to involve all elements of $*$-ring.  Mitsch
\cite{mt} defined a partial order on a semigroup. We modify that order to have partial order on a $*$-ring.

In this paper, we introduce a partial order on a $*$-ring which is a generalization of the partial order on the set of projections in a Rickart $*$-ring. For a $*$-ring $R$, it is proved that the poset  $(R, \leq)$ is an sectionally semi-complemented (SSC) poset. For an abelian Rickart $*$-ring, we prove that every interval $[0,x]$ is an orthomodular poset, in fact, an orthomodular lattice. In the last section, comparability axioms are introduced to involve all elements of the $*$-ring. 
\section {natural partial order  and its properties}
We introduce an order on a $*$-ring with unity.
\begin{defn}\label{d1}Let $R$ be a $*$-ring with unity. Define a
 relation `$\leq$' on $R$ by `$a\leq b$ if and only if $a=xa=xb=ax^*=bx^*$,  for
 some $x\in R$'.
\end{defn}

\begin{prop}
Let $R$ be a $*$-ring with unity. Then the relation `$\leq$'
given in Definition \ref{d1} is a partial order on $R$.
\end{prop}
\begin{proof}Reflexive: for $x=1$, we have
 $a=xa=ax^*$. Hence $a\leq a, \forall a\in R$.\\
 Antisymmetric: Let $a\leq b$ and $b\leq a$. Then there
 exist $x,y \in R$ such that $a=xa=xb=ax^*=bx^*$ and
 $b=yb=ya=by^*=ay^*$, hence $b=ya=y(bx^*)=bx^*=a$.\\
Transitive: Let $a\leq b$ and $b\leq c$. Hence there
exist $x,y \in R$ such that $a=xa=xb=ax^*=bx^*$ and
$b=yb=yc=by^*=cy^*$. Then $(xy)a=(xy)(bx^*)=x(yb)x^*=xbx^*=ax^*=a,
(xy)c=x(yc)=xb=a,
a(xy)^*=(xb)(y^*x^*)=x(by^*)x^*=xbx^*=ax^*=a$ and
$c(xy)^*=c(y^*x^*)=(cy^*)x^*=bx^*=a$. Hence $a\leq c$.
 \end{proof}

{\bf Henceforth $\mathbf{R}$ denotes a $*$-ring with unity and we say that $a\leq b$ \textit{through} $x$ whenever $a=xa=xb=ax^*=bx^*$}.

 \begin{note} If we restrict this partial order to the set of projections in a Rickart $*$-ring,
 then it coincides with the usual partial order for projections given in Berberian \cite{skb}.
\end{note}
 
 \begin{rem}This order is different from $*$-order.
 	
\vspace{.1cm}

\noindent
  For, let $a={\displaystyle\left[%
\begin{array}{cc}
  1 & 2 \\
  1 & 2 \\
\end{array}%
\right], b=\left[%
\begin{array}{cc}
  2 & 0 \\
  0 & 1 \\
\end{array}%
\right]}\in R=M_2(\mathbb{Z}_3)$ with transpose as an involution.\\ 

\vspace{.1cm}
\noindent Then ${\displaystyle a^*a=\left[%
\begin{array}{cc}
  2 & 1 \\
  1 & 2 \\
\end{array}%
\right]}=a^*b$, $aa^*={\displaystyle\left[%
\begin{array}{cc}
  2 & 2 \\
  2 & 2 \\
\end{array}%
\right]=ba^*}$, hence $a\underset{*}{\leqslant} b$.\\

\vspace{.1cm}
\noindent
Next let $x={\displaystyle\left[%
\begin{array}{cc}
  x_1 & x_2 \\
  x_3 & x_4 \\
\end{array}%
\right]}$ be such that $a=xa=xb=ax^*=bx^*$. Then $a=xa$ gives

\vspace{.1cm}
${\displaystyle \left[%
\begin{array}{cc}
  1 & 2 \\
  1 & 2 \\
\end{array}%
\right]}={\displaystyle\left[%
\begin{array}{cc}
  x_1 & x_2 \\
  x_3 & x_4 \\
\end{array}%
\right]}{\displaystyle\left[%
\begin{array}{cc}
  1 & 2 \\
  1 & 2 \\
\end{array}%
\right]}={\displaystyle\left[%
\begin{array}{cc}
  x_1+x_2 & 2(x_1+x_2) \\
  x_3+x_4 & 2(x_3+x_4) \\
\end{array}%
\right]}$ and $a=ax^*$ gives \\ 

\vspace{.1cm}
${\displaystyle \left[%
\begin{array}{cc}
  1 & 2 \\
  1 & 2 \\
\end{array}%
\right]=\left[%
\begin{array}{cc}
  1 & 2 \\
  1 & 2 \\
\end{array}%
\right]\left[%
\begin{array}{cc}
  x_1 & x_3 \\
  x_2 & x_4 \\
\end{array}%
\right]=\left[%
\begin{array}{cc}
  x_1+2x_2 & x_3+2x_4 \\
  x_1+2x_2 & x_3+2x_4 \\
\end{array}%
\right]}$.\\

On comparing, we get $x_1+x_2=1=x_1+2x_2$, which gives $x_2=0, x_1=1$.
Similarly $x_3+x_4=1,x_3+2x_4=2$, giving $x_3=1, x_4=0$, {\it i.e.}, 
\vspace{.1cm}
$x=\left[%
\begin{array}{cc}
  1 & 0 \\
  1 & 0 \\
\end{array}%
\right]$. But $xb\neq a$. Hence $a\nleq b$. On the Other hand, if $c=\left[%
\begin{array}{cc}
  1 & 1 \\
  1 & 1 \\
\end{array}%
\right], d=\left[%
\begin{array}{cc}
  0 & 1 \\
  1 & 1 \\
\end{array}%
\right]$ and $y=\left[%
\begin{array}{cc}
  0 & 1 \\
  0 & 1 \\
\end{array}%
\right]$. Then  $c \leq d$ through $y$. While $c^*c\neq c^*d$, hence $c\underset{*}{\nleqslant} d$. Thus these two partial orders (natural partial order and $*$-order) are distinct.
\end{rem}

\begin{prop} Let $R$ be a commutative $*$-ring. Then $a\leq b$ implies $a\underset{*}{\leqslant} b$.
\end{prop}
\begin{proof} Let $a\leq b$. Then there exists $x\in R$ such that
$a=xa=xb=ax^*=bx^*$. This yields
$a^*b=(xa)^*b=a^*x^*b=a^*(bx^*)=a^*a$ and  since $R$ is commutative, we get $aa^*=ba^*$. Hence $a\underset{*}{\leqslant} b$.
 \end{proof}
 Note that the converse of the above statement is not true in general. Since $*$-order is not a partial order on $\mathbb{Z}_{12}$ with identity mapping as an involution(as $6^*6=6^*0=0^*0$).
 
In the next result, we provide properties of the natural partial order.
\begin{thm}\label{t1}Let $R$ be a $*$-ring with unity. Then the following statements hold.
\begin{enumerate}
\item 0 is the least element of the poset $R$.
\item $a \leq e, a\in R, e\in P(R)$(set of projections in $R$) implies $a\in P(R)$.
\item $a\leq b$ if and only if $a^*\leq b^*$.
\item If $a\leq b$, then $r(b)\subseteq r(a)$ and $l(b)\subseteq l(a)$.
\item $a \leq b,$ $b$ regular ({\it i.e.}, $bb'b=b$, for some $b'\in R$) implies $a$ is regular.
\item $a\leq b$ and $a$ has right(resp. left) inverse
imply $a=b$, {\it i.e.}, every element having right(resp. left) inverse is
maximal. 
\item If $a\leq b$. Then $ac\leq bc$ if and only if $ca\leq
cb, \forall c\in R$.
\end{enumerate}
\end{thm}
\begin{proof}$(1)$ Obvious.\\
$(2)$ Suppose $a\leq e, e\in P(R)$. Let $a\leq e$ through $x$, for some $x\in
R$, {\it i.e.}, $a=xa=xe=ax^*=ex^*$. This yields $a^2=xe.ex^*=xex^*=ax^*=a$. Also,
$a^*=(xe)^*=ex^*=a$, hence $a\in P(R).$\\
$(3)$ Let $a\leq b$. Then $a=xa=xb=ax^*=bx^*$, for some $x\in
R$. Hence $a^*=xa^*=a^*x^*=b^*x^*=xb^*$ which gives $a^*\leq b^*$. The Converse follows from the fact that $(a^*)^*=a$.\\
$(4)$ Obvious.\\
$(5)$ Suppose $a\leq b$ and $b$ is regular, {\it i.e.}, $bb'b=b$, for some $b'\in R$. Let $a=xa=xb=ax^*=bx^*$, for some
$x\in R$. Then  $a=ax^*=xbx^*=xbb'bx^*=(xb)b'(bx^*)=ab'a$. Hence $a$ is regular.\\
$(6)$ Let $c\in R$ be such that $ac=1$(resp. $ca=1)$ and $a\leq b$. Let $a=xa=xb=ax^*=bx^*$, for some $x\in R$. Then
$a=xa$(resp. $a=ax^*)$ gives $ac=xac$(resp.
$ca=cax^*)$. Thus $1=x$(resp. $1=x^*)$. Hence $a=b$, i.e, $a$ is a maximal element.\\ 
$(7)$ Suppose $a\leq b$ and $ac\leq bc,$ $\forall c\in R$. Since $a\leq b$, by  ($3$) above, we have
$a^*\leq b^*$ giving $a^*c^*\leq b^*c^*$, which further yields $(a^*c^*)^*\leq (b^*c^*)^*$,  {\it i.e.},  $ca\leq
 cb$ and conversely.
\end{proof}

 In a poset $P$, the {\it principal ideal} generated by $a\in P$ is given by $(a]=\{x\in P\colon x\leq a\}$.
\begin{prop}If $a$ and $b$ are central elements of a
 $*$-ring $R$ which generate the same ideals of a ring $R$, then there is a order isomorphism between the set of elements $\leq a$ and the set of elements $\leq b$.
 \end{prop}
\begin{proof}Let $a$ and $b$ are central elements with $Ra=Rb$. Then $a=bs, b=at$, for some $s,t\in R$. Denote $(a]=\{x\in R\colon x\leq a\}$. Define $\phi : (a] \rightarrow (b]$ by $\phi(x)=xt$. We claim that $xt\leq b, \forall x\in (a]$.
 As $x\leq a$, we have $x=x_1x=x_1a=ax_1^*=xx_1^*$, for some $x_1\in R$. Then
 $x_1xt=xt$,
 $xtx_1^*=x_1atx_1^*=x_1bx_1^*=x_1x_1^*b=x_1x_1^*at=x_1ax_1^*t=x_1xt=xt$,
$x_1b=x_1at=ax_1t=xt$ and $bx_1^*=x_1^*b=x_1^*at=ax_1^*t=xt$.
Hence $xt\leq b$. Now, let $x,y\in (a]$ be such that
$x=x_1x=x_1a=ax_1^*=xx_1^*$ and $y=y_1y=y_1a=ay_1^*=yy_1^*$, for
some $x_1,y_1\in R$. Then
 $\phi(x)=\phi(y)$ if and only if $xt=yt$ if and only if $x_1at=y_1at$ if and only if $x_1b=y_1b $ if and only if $x_1a=y_1a $
 if and only if $x=y$. Hence $\phi$ is well defined and one-one.
 Let $z\in (b]$. Then as above $zs\in (a]$ and $z=z_1b=z_1z=bz_1^*=zz_1^*$, for some $z_1\in
 R$. Also $\phi(zs)=zst=z_1bst=z_1at=z_1b=z$, {\it i.e.}, $\phi$ is a
 bijection.
 
 Now, suppose that $x,y\in(a]$ with $x\leq y$. Then $x=x_1x=x_1a=ax_1^*=xx_1^*$, $y=y_1y=y_1a=ay_1^*=yy_1^*$ and $x=x_2x=x_2y=yx_2^*=xx_2^*$,
 for some $x_1,x_2,y_1\in R$.
Next, $(x_1x_2)xt=x_1xt=xt$, $(x_1x_2)yt=x_1xt=xt$,
$xt(x_1x_2)^*=xtx_2^*x_1^*=x_1atx_2^*x_1^*=x_1bx_2^*x_1^*=x_1x_2^*x_1^*b=x_1x_2^*x_1^*at=ax_1x_2^*x_1^*t=xx_2^*x_1^*t=xx_1^*t=xt$
and
$yt(x_1x_2)^*=ytx_2^*x_1^*=y_1atx_2^*x_1^*=y_1bx_2^*x_1^*=y_1x_2^*x_1^*b=y_1x_2^*x_1^*at=y_1ax_2^*x_1^*t=yx_2^*x_1^*t=xx_1^*t=xt$. Consequently  $\phi(x)\leq \phi(y)$. Hence $\phi$ is an order isomorphism.
In fact, $\psi:(b]\rightarrow (a]$ defined by $\psi(y)=ys$, works
as an inverse of $\phi$.
 \end{proof}
 
 \begin{theorem}[Condition of Compatability] If $xa=ax^*$,  $\forall a,x\in R$, then the natural
 	partial order is compatible with multiplication, {\it i.e.}, $a\leq b$ implies $ca\leq cb$, for all $c\in R$.
 \end{theorem}
 \begin{proof}In view of Theorem \ref{t1} (7), it is enough to show that $a\leq b$ implies $ac\leq bc, \forall c\in
 	R$. Let $a\leq b$, then there exists $x\in R$ such that  $a=xa=xb=ax^*=bx^*$.
 	Hence $ac=xac=xbc=ax^*c=bx^*c$, {\it i.e.}, $ac=xac$,
 	$acx^*=ca^*x^*=c(xa)^*=ca^*=ac$. Also $bcx^*=cb^*x^*=c(xb)^*ca^*=ac$, hence $ac\leq bc$.
 \end{proof}

\begin{defn}\label{d2} Two elements $a$ and $b$ in a $*$-ring $R$
are {\it orthogonal}, denoted by $a\perp b$, if there exists $x\in R$ such that
$xa=a=ax^*$ and $xb=0=bx^*$.
\end{defn}

The orthogonality relation in a $*$-ring has the following properties.
\begin{thm}\label{t2}
For elements $a,b,c$ in a $*$-ring $R$ , the following statements hold.
\begin{enumerate}
\item $a\perp a$ implies $a=0$.
\item $a\perp b$ if and only if $b\perp a$ if and only if $a\perp (-b)$.
\item $a\perp b$, $c\leq a$ imply $c\perp b$.
\item $a\perp b $ if and only if $a\leq a-b$.
\item $a\leq b$ implies $b-a\leq b$ and $b-a\perp a$.
\item If $a\perp b$, then $a\wedge b=0$ and $a+b$ is an upper bound of both $a,b$.
\item $a\perp b$, $(a+b)\perp c$ imply $a\perp (b+c)$.
\end{enumerate}
\end{thm}

\begin{proof}$(1), (2)$ Obvious.\\
$(3)$ Suppose that $a\perp b$ and $c\leq a$. Let $x, y\in R$ be such that $a=xa=ax^*$, $xb=0=bx^*$ and  $c=yc=cy^*=ya=ay^*$. Then $(yx)c=yxay^*=yay^*=cy^*=c$. Similarly, $c(yx)^*=c$. On the other hand, $(yx)b=0$ and $b(yx)^*=0$. Consequently, $c\perp b$.\\
$(4)$ Suppose $a$ and $b$ are orthogonal. Let $x\in R$ be such that $a=xa=ax^*$ and $xb=0=bx^*$. Then
$a=x(a-b)=(a-b)x^*=xa=ax^*$, hence $a\leq a-b$. Conversely, suppose that  $a\leq a-b$. Let $x\in R$ be such that $a=x(a-b)=(a-b)x^*=xa=ax^*$. Then $a=x(a-b)$ and $a=xa$ gives
$xb=0$. Similarly, $bx^*=0$. Hence $a\perp b$.\\
$(5)$ Let $x\in R$ be such that $a=xa=xb=ax^*=bx^*$. Then $(1-x)(b-a)=b-a-xb+xa=b-a-a+a=b-a$,
$(1-x)b=b-xb=b-a$, $b(1-x)^*=b-bx^*=b-a$ and
$(b-a)(1-x)^*=b-a-bx^*-ax^*=b-a-a+a=b-a$. Hence $b-a\leq b$. Also
$(1-x)(b-a)=b-a=(b-a)(1-x)^*$ and $(1-x)a=0=a(1-x)^*$. Hence $b-a \perp a$.\\
$(6)$ Suppose $a\perp b$ and $x$ be such that
$xa=a=ax^*, xb=0=bx^*$. Let $c\leq a$ and $c\leq b$ {\it i.e.}
$c=x_1c=x_1a=cx_1^*=ax_1^*$ and $c=x_2c=x_2b=cx_2^*=bx_2^*$, for some $x_1, x_2 \in R$. Then
$x_1a=c=x_2b$ gives $c=x_1a=x_2b=x_1ax^*=x_2bx^*=0$. Hence $a\wedge b=0$. From $(2)$ and $(4)$, we have $a\leq a+b$ and $b=(a+b)-a\leq a+b$.\\
$(7)$ Suppose that $a\perp b$, $(a+b)\perp c$. From $(6)$, we have $a\leq a+b$ and $a+b\leq a+b+c$. This gives $a\leq a+b+c$. Then from $(5)$, we get $b+c=(a+b+c)-a\leq a+b+c$ and $(b+c)\perp a$, as required.
\end{proof}

A poset $P$ with 0 is called {\it sectionally semi-complemented} (in brief SSC) if, for $a,b\in P$, $a< b$, there exists
an element $c\in P$ such that $0<c< b$ and $\{a,c\}^l=\{	0\}$, where $\{a,c\}^l=\{x\in P\colon x\leq a\textnormal{ and } x\leq c\}$. Thus from $(5)$ and $(6)$ of Theorem \ref{t2} , we have the following result.
 \begin{thm}
 	Let $R$ be a $*$-ring. Then the poset $(R, \leq)$ is an SSC poset.
 \end{thm}

A ring is called  an {\it abelian ring} if all of its idempotents are central.
\begin{prop}\label{p2}In an abelian Rickart $*$-ring $R$, the following statements are equivalent.
\begin{enumerate}
\item[i)] $a\leq b$.
\item[ii)]  There exists a projection $e$ such that
$a=ae=be$.
\item [iii)] $ab=a^2=ba$.
\end{enumerate}
\end{prop}
\begin{proof}i)$\implies$ii) Suppose $a\leq b$, then there exists
$x\in R$ such that $a=xa=xb=ax^*=bx^*$. Since $ a=xa$, we have  $(1-x)a=0$. This gives $ a\in r\{1-x\}=eR$, for some projection $e\in
R$. Then $ea=a=ae$ and $(1-x)e=0$, {\it i.e.}, $ e=xe=ex^*$. Also,
$a=xb$ implies $ea=exb=xeb=eb$. Thus $a=ae=be$.\\
ii)$\implies$iii) Obvious.\\
iii)$\implies$i) Let $ab=a^2 $, {\it i.e.}, $a(b-a)=0$. Then there exists a
projection $e$ such that $a=ea=ea$ and $e(b-a)=0$, {\it i.e.}, $eb=ea=a$,
hence $a\leq b$.
\end{proof}
\begin{lem}\label{l6} If $R$ is an abelian Rickart $*$-ring, then $a\perp
b$ implies $a\wedge b=0$ and $a\vee b=a+b$.
\end{lem}
\begin{proof}Suppose $a\perp b$. By Theorem \ref{t2} $(6)$,  $a\wedge b=0$ and $a+b$ is an upper
bound of $a$ and $b$. Let $a\leq c$ and $b\leq c$, then there
exist projections $e,f$ such that $a=ea=ec$ and $b=fb=fc$. Since
$a\perp b$,  there exists  $x\in R$ such that $xa=a=ax^*$
and $xb=0=bx^*$. Let $y=ex+f(1-x)$. Then
$y(a+b)=exa+exb+f(1-x)a+f(1-x)b=a+b$, $(a+b)y^*=a+b$,
$yc=exc+f(1-x)c=a+b$ and $cy^*=a+b$, {\it i.e.}, $a+b\leq c$. Thus
$a\vee b=a+b.$
\end{proof}

Before proceeding further, we need the
definition of orthomodular poset.  An {\it orthomodular
poset} is a partially ordered set $P$ with 0 and 1 equipped with a mapping $x\rightarrow x^{\perp}$
(called the {\it orthocomplementation}) satisfying the conditions.
\begin{enumerate}
\item[i)] $a\leq b\Rightarrow b^{\perp}\leq a^{\perp}$,
\item[ii)] $(a^{\perp})^{\perp}=a$ for all $a\in P$,
\item[iii)] $a\vee a^{\perp}=1$ and $a\wedge a^{\perp}=0,$ for all $a\in P$,
\item[iv)] $a\leq b^{\perp}$ implies that $a\vee b$ exists in $P$,
\item[v)] $a\leq b \Rightarrow b=a\vee (a\vee b^{\perp})^{\perp}$.
\end{enumerate}

The following result is essentially due to Marovt et al. {\cite[Theorem 1]{jm}}.
\begin{thm}\label{t3}Let $R$ be a Rickart $*$-ring. Then $a\underset{*}{\leqslant} b$ if and only if there exist projections $p$ and $q$ such that $a=pb=bq$.
\end{thm}

Thus, from Proposition \ref{p2} and Theorem \ref{t3}, the natural partial order and $*$-order are equivalent on abelian Rickart $*$-rings. This leads to the following two results which are also proved independently by Janowitz \cite{j}.
\begin{theorem} Let $R$ is an abelian Rickart $*$-ring. Then
every interval $[0,x]$ is an orthomodular poset.
\end{theorem}
 We know that, if $R$ is a Rickart $*$-ring, then the set of projection
$P(R)$ forms a lattice and the set $\{e\in P(R)\colon e\leq x''\}$ is
 sub lattice of $P(R)$, where $x'$ is a projection which generates
 the right annihilator of $x$.
\begin{theorem} In an abelian Rickart $*$-ring $R$ every interval
 $[0,x]$ is ortho-isomorphic to $\{e\in P(R)\colon e\leq x''\}.$ Hence
 every interval $[0,x]$  is an orthomodular lattice.
 \end{theorem}

\section{comparability axioms}
 Two
projections $e$ and $f$ are said to be {\it equivalent}, written $e\sim
f$, if there is $w\in eRf$ such that $e=ww^*$ and $f=w^*w$, which is
an equivalence relation on the set of projections in a Rickart
$*$-ring. A projection $e$ is said to be \textit{dominated} by the projection
$f$, denoted by $e\lesssim f$, if $e\sim g\leq f$, for some
projection $g$ in $R$. Two projections $e$ and $f$ are said to be
\textit{generalized comparable} if there exists a central projection $h$ such
that $he\lesssim hf$ and $(1-h)f\lesssim (1-h)e$. A $*$-ring is said
to satisfy the {\it generalized comparability} $(GC)$ if any two
projections are generalized comparable. Two projections $e$ and $f$
are said to be {\it partially comparable} if there exist non zero
projections $e_0$, $f_0$ in $R$ such that $e_0\leq e$, $f_0\leq f$
and $e_0\sim f_0$. If for any pair of projections in $R$, $eRf\neq
0$ implies $e$ and $f$ are partially comparable, then $R$ is said to
satisfy \textit{partial comparability} ($PC$). More about comparability axioms on the set of projections in a Rickart $*$-ring can be found in Berberian \cite{skb}.

Drazin \cite{d} extended the relation of equivalence of two projections to general elements of a $*$-ring as follows.
\begin{defn}[{\cite[Definition 2*]{d}}]Let $R$ be a $*$-ring with unity. We say that
$a\sim b$ if and only if there exists $x\in aRb,y\in bRa$ 
such that $aa^*=xx^*, bb^*=yy^*, a^*a=y^*y,b^*b=x^*x$.
\end{defn}
This relation is symmetric on a $*$-ring.  Thakare and Nimbhorkar \cite{nsk} extended the comparability axioms using the above relation and $*$-order to involve all elements of Rickart $*$-ring. 


We provide a relation which is symmetric and transitive on general elements of $*$-ring as an extension of the relation of equivalence of two projections.

\begin{defn}Let $R$ be a $*$-ring with unity. We say that
	$a\sim b$ if and only if there exists $x,y\in R$ such that $aa^*=xx^*, bb^*=yy^*, a^*a=y^*y,b^*b=x^*x$ with $x=ax=xb$ and $y=by=ya$.
\end{defn}
 
Now, we extend the concepts of dominance, $GC, PC$ etc. from the set of projections in a Rickart
$*$-ring to general elements in a $*$-ring.
\begin{defn}
\begin{enumerate}
\item Let $R$ be a $*$-ring with unity. We say that
{\it $a$ is dominated by $b$} if $a\sim c\leq b$ for some $c\in R$. In
notation $a\lesssim b$.
\item A $*$-ring $R$ is said to satisfy the {\it generalized comparability for elements} ($GC$) for elements, if for
any $a,b \in R$ there exists a central projection $h$ such that
$ha\lesssim hb$ and $(1-h)b\lesssim (1-h)a$.
\item Two elements $a,b$ in
a $*$-ring $R$ are said to be {\it partially comparable} if there exists
two non-zero elements $c,d$ in $R$ such that $c\leq a, d\leq b$
with $c\sim d$.
 If for any $a,b \in R$, $aRb\neq 0$ implies $a$ and $b$ are
 partially comparable then we say that $R$ has {\it partial comparability for elements} $(PC)$.
\end{enumerate}
\end{defn}

Clearly, if $a\leq b$ or $a\sim b$, then $a\lesssim b$.



\begin{lem} If $a\lesssim b$ and $h$ is a central projection,
then $ha\lesssim hb$.
\end{lem}

 \begin{defn} Two elements $a$ and $b$ in a $*$-ring $R$
 are said to be {\it very orthogonal} if there exists a central
 projection $h$ such that $ha=a$ and $hb=0$.
 \end{defn}
 
 The relevance of very orthogonality to generalized comparability is as follows:
\begin{theorem}\label{gc} If $a$ and $b$ are elements of a $*$-ring $R$. Then the 
following statements are equivalent.
\begin{enumerate}
\item[i)]$a$ and $b$ are generalized comparable.
\item[ii)] There exists orthogonal decompositions $a=x+y$, $b=z+w$ with
$x\sim z$, $y $ and $w$ are very orthogonal.
\end{enumerate}
\end{theorem}
\begin{proof}i) $\Rightarrow$ ii) Suppose $a$ and $b$ are generalized comparable. Let $h$ be a central projection such that $ha\lesssim hb$ and
$(1-h)b\lesssim (1-h)a$. Then $ha\sim k_1\leq hb$, $(1-h)b\sim
k_2\leq (1-h)a$, for some $k_1,k_2\in R$. Hence $k_1=m_1k_1=m_1hb=k_1m_1^*=hbm_1^*,$ for
some $m_1\in R$. Then $k_1=m_1hb$ gives $k_1h=m_1hbh=m_1hb=k_1$. Similarly, $k_2=(1-h)k_2$. Also
$hak_2^*=ha(1-h)k_2^*=0=(ha)^*k_2$, $(1-h)bk_1^*=[(1-h)b]^*k_1=0$. 

We claim that $ha+k_2\sim k_1+(1-h)b$. Since $ha\sim k_1$, there exist $x_1,y_1\in R$ such that $(ha)(ha)^*=x_1x_1^*,~ k_1k_1^*=y_1y_1^*,~(ha)^*(ha)=y_1^*y_1$ and $k_1^*k_1=x_1^*x_1$ with $x_1=hax_1=x_1k_1$ and $y_1=k_1y_1=y_1ha$. Clearly, $x_1=hx_1$ and $y_1=hy_1$, since $k_1h=k_1$. Similarly, 	Since $(1-h)b\sim k_2$, there exist $x_2,y_2\in R$ such that $k_2k_2^*=x_2x_2^*,~ [(1-h)b][(1-h)b]^*=y_2y_2^*,~k_2^*k_2=y_2^*y_2$ and $[(1-h)b]^*[(1-h)b]=x_2^*x_2$ with $x_2=k_2x_2=x_2(1-h)b$ and $y_2=(1-h)by_2=y_2k_2$. Clearly, $x_2=(1-h)x_2$ and $y_2=(1-h)y_2$, since $k_2(1-h)=k_2$.
 	
 	Let $x=x_1+x_2$ and $y=y_1+y_2$. Since $hk_2=0$ and $(1-h)k_1=0$, we have  $(ha+k_2)x=(ha+k_2)(x_1+x_2)=hax_1+hax_2+k_2x_1+k_2x_2=x_1+0+0+x_2=x$, $x[k_1+(1-h)b]=(x_1+x_2)[k_1+(1-h)b]=x_1k_1+x_1(1-h)b+x_2k_1+x_2(1-h)b=x_1+0+0+x_2=x$. Similarly, we have $y=[k_1+(1-h)b]y=y(ha+k_2)$.
 	
 	Also, $xx^*=(x_1+x_2)(x_1+x_2)^*=x_1x_1^*+x_1x_2^*+x_2x_1^*+x_2x_2^*=x_1x_1^*+0+0+x_2x_2^*=(ha)(ha)^*+k_2k_2^*=[ha+k_2][ha+k_2]^*$ and  $x^*x=(x_1+x_2)^*(x_1+x_2)=x_1^*x_1+x_1^*x_2+x_2^*x_1+x_2^*x_2=x_1^*x_1+0+0+x_2^*x_2=k_1^*k_1+[(1-h)b]^*[(1-h)b]=[k_1+(1-h)b]^*[k_1+(1-h)b]$. On the other hand, $yy^*=(y_1+y_2)(y_1+y_2)^*=y_1y_1^*+y_1y_2^*+y_2y_1^*+y_2y_2^*=k_1k_1^*+0+0+[(1-h)b][(1-h)b]^*=[k_1+(1-h)b][k_1+(1-h)b]^*$ and $y^*y=(y_1+y_2)^*(y_1+y_2)=y_1^*y_1+y_1^*y_2+y_2^*y_1+y_2^*y_2=y_1^*y_1+0+0+y_2^*y_2=(ha)^*(ha)+k_2^*k_2=[ha+k_2]^*[ha+k_2]$. Therefore $ha+k_2\sim k_1+(1-h)b$.

Next, we claim that $ha+k_2\leq a$ and $k_1+(1-h)b\leq b$. Since $h$ is a central projection, $k_2\leq (1-h)a\leq a$ and $ha \leq
a$, implies $ k_2=x_1k_2=x_1a=k_2x_1^*=ax_1^*$  and
$ha=x_2ha=x_2a=hax_2^*=ax_2^*$, for some $x_1,x_2\in R$. Let
$y_1=x_1+hx_2$, then
$y_1(ha+k_2)=x_1ha+x_1k_2+hx_2ha+hx_2k_2=ha+k_2$,
$(ha+k_2)y_1^*=hax_1^*+k_2x_1^*+hax_2^*h+k_2x_2^*h=ha+k_2$ and
$y_1a=a(x_1+hx_2)^*=ha+k_2=ay_1^*$, therefore $ha+k_2\leq a$.
Similarly, $(1-h)b+k_1\leq b$. Now put $ha+k_2=x$, $(1-h)b+k_1=z$,
$y=a-x$ and $w=b-z$ implies $hb-k_1=b-z=w$. Then
$hw=h(hb-k_1)=hb-k_1=w$ and $hy=h(a-x)=ha-hx=ha-ha-hk_2=0$ (since $hk_2=0$), {\it i.e.}, $y$
and $w$ are very orthogonal. Thus $a=x+y, b=z+w$ where $x\perp y$,
$z\perp w$ such that we get $x\sim z$ with $y$ and  $w$ are very orthogonal.\\
ii) $\Rightarrow$ i) Let $h$ be a central projection such that $hw=w$ and $hy=0$. Then
$ha=hx+hy=hx$ and $(1-h)b=(1-h)z+(1-h)w=(1-h)z$, where $ha=hx\sim
hz\leq hb$ and $(1-h)b=(1-h)z\sim (1-h)x\leq (1-h)a$. Thus
$ha\lesssim hb$ and $(1-h)b\lesssim (1-h)a$. Hence $a,b$ are
generalized comparable.
 \end{proof}

Next result implies that $GC$ for elements is stronger than $PC$ for elements.
\begin{thm}\label{t4} If $R$ is a $*$-ring with $GC$ for elements then it has $PC$ for elements.
\end{thm}
\begin{proof} Let $a,b$ are elements of $R$ which are not
partially comparable. We will show that $aRb=0$. Applying $GC$ to
the pair $a,b$ we get orthogonal decompositions $a=x+y$ and
$b=z+w$, where $x\sim z$ and $y,w$ are very orthogonal. If $x\neq
0$ and $w\neq 0$ then $a$ and $b$ are partially comparable, which
is a contradiction to the assumption. Hence $x=0=w$, {\it i.e.}, $a,b$ are
very orthogonal. Let $h$ be a central projection such that $ha=a$
and $hb=0$. Then $aRb=haRb=aRhb=0$. Thus $R$ has $PC$ for elements.
 \end{proof}
 
 \begin{lem}
 In an abelian Rickart $*$-ring $a \perp b$ if and only if $RP(a)RP(b)=0$.
 \end{lem}
\begin{proof}First we show that $ab=0$ if and only if $RP(a)RP(b)=0$. Suppose that $ab=0$ which gives $b\in r(\{a\})=(1-RP(a))R$. Hence $(1-RP(a))b=b$ giving $RP(a)b=0$. Since all projections in $R$ are central, we get $RP(a)\in r(\{b\})=(1-RP(b))R$. Which yields $RP(b)RP(a)=0$. Conversely, if $RP(a)RP(b)=0$, then $ab=(aRP(a))(bRP(b))=aRP(a)RP(b)b=0$.

Next, Suppose that $a\perp b$. Then there exists $x\in R$ such that $xa=a=ax^*$ and $xb=0=bx^*$, {\it i.e.},  $a(1-x^*)=0$. Hence $RP(a)RP(1-x^*)=0$. Since $R$ is abelian, we have $RP(1-x^*)=1-RP(x^*)=1-RP(x)$. Consequently, $RP(a)RP(x)=RP(a)$. On the other hand, $xb=0$ implies $RP(x)RP(b)=0$. Then $RP(a)RP(b)=RP(a)RP(x)RP(b)=0$, hence $ab=0$. Conversely, if $ab=0$, then $RP(a)RP(b)=0$. Thus $RP(a)a=a=aRP(a)$ and $RP(a)b=0=bRP(a)$. Hence $a\perp b$.
\end{proof}
The next result shows that the relation $\sim$ is finitely additive. 
\begin{thm}
Let $R$ be an abelian Rickart $*$-ring. If $a_1\perp a_2,~ b_1\perp b_2$ with $a_1\sim b_1$ and $a_2\sim b_2$, then $a_1+a_2\sim b_1+b_2$, {\it i.e.}, the relation $\sim $ is finitely additive.
\end{thm}
\begin{proof}Since $a_1\perp a_2,~ b_1\perp b_2$, we have $RP(a_1)RP(a_2)=0=RP(b_1)RP(b_2)$. Also, $a_1\sim b_1$ and $a_2\sim b_2$ there exists $x_i, y_i\in R$ such that $a_ia_i^*=x_ix_i^*,~ a_i^*a_i=y_i^*y_i,~b_ib_i^*=y_iy_i^*,~ b_i^*b_i=x_ix_i$ with $x_i=a_ix_i=x_ib_i$ and $y_i=b_iy_i=y_ia_i$ for $i=1,2$. This gives $x_i(1-a_i)=0$(since in an  abelian Rickart $*$-ring $RP(x)=LP(x)$), hence $RP(x_i)=RP(x_i)RP(a_i)$, for $i=1,2$. Then for $i\neq j$, we have $x_ia_j=x_iRP(x_i)a_jRP(a_j)=x_iRP(x_i)RP(a_i)a_jRP(a_j)=x_iRP(x_i)RP(a_1)RP(a_j)a_j=0$. Moreover $x_ix_j^*=0=x_i^*x_j$ for $i\neq j$. Similarly,  we have $b_jx_i=0$ for $i\neq j$.

Let $x=x_1+x_2$ and $y=y_1+y_2$. Then $(a_1+a_2)x=a_1x_1+a_2x_1+a_1x_2+a_2x_2=x_1+0+0+x_2$ and   $(b_1+b_2)x=b_1x_1+b_1x_2+b_2x_1+b_2x_2=x_1+0+0+x_2=x$. Consider $xx^*=x_1x_1^*+x_2x_1^*+x_1x_2^*+x_2x_2^*=a_1a_1^*+0+0+a_2a_2^*=(a_1+a_2)(a_1+a_2)^*$ and $x^*x=x_1^*x_1+x_2^*x_1+x_1^*x_2+x_2^*x_2=b_1^*b_1+b_1^*b_2=(b_1+b_2)^*(b_1+b_2)$. Similarly, $y=(b_1+b_2)y=y(a_1+a_2)$, $yy^*=(b_1+b_2)(b_1+b_2)^*$ and $y^*y=(a_1+a_2)^*(a_1+a_2)$. Therefore $a_1+a_2\sim b_1+b_2$.
\end{proof}
Above result ensures that the converse of Theorem \ref{t4} is true for finite abelian Rickart $*$-rings.
\begin{thm}
	Let $R$ be a finite abelian Rickart $*$-ring. Then $GC$ for elements and $PC$ for elements are equivalent.
\end{thm}
\begin{proof}Suppose that $R$ has $PC$ for elements. It is enough to show that, $PC$ for elements implies $GC$ for elements. Let $a,b\in R$. If $aRb=0$, then $ab=0$. This gives $RP(a)b=0$. Since $R$ is an abelian ring, we get $a$ and $b$ are very orthogonal.  Hence we are done. Suppose $aRb\neq 0$.  Hence there exist $a_0\leq a$ and $b_0\leq b$ such that $a_0\sim b_0$. Let $a_1, b_1$ be the largest elements such that $a_1\leq a$, $b_1\leq b$ and $a_1\sim b_1$. Then $a_2=a-a_1$ and $b_2=b-b_1$ are such that $a_2\leq a$, $b_2\leq b$, $a_1\perp a_2$ and $b_1\perp b_2$. By the maximality of $a_1$ and $b_1$, we get $a_2Rb_2=0$, which gives $a_2$ and $b_2$ very orthogonal. Thus we get an orthogonal decompositions $a=a_1+a_2$, $b=b_1+b_2$ such that $a_1\sim b_1$, $a_2$ and $b_2$ very orthogonal. By Theorem \ref{gc} we have $a$ and $b$ are generalized comparable.
	\end{proof}

\begin{prop}Let $R$ be a $*$-ring with $GC$ for elements and $e$ is any projection in $R$. Then $eRe$ also has $GC$ for elements.
\end{prop}
\begin{proof} Let $a,b\in eRe\subseteq R$. Then there exists a
central projection $h$ in $R$ such that $ha\lesssim hb$,
$(1-h)b\lesssim (1-h)a$. Let $g=ehe=he\in eRe$ and $x$ be any
element in $eRe$. Then $gx=hex=hx=xh=xeh=xhe=xg$. Hence $g$ is a
central projection in $eRe$ with $ga=hea=ha, gb=heb=hb$, {\it i.e.},
$ga\lesssim gb$ and $(e-g)b=ab=hab=b-hb=(1-h)b$,
$(e-g)a=ea-hea=a-ha=(1-h)a$, {\it i.e.}, $(e-g)b\lesssim (e-g)a$. Thus
$a$ and $b$ are generalized comparable in $eRe$.
 \end{proof}

\begin{cor} If the matrix ring $M_n(R)$ has $GC$ for elements, then $R$ has $GC$ for elements.
\end{cor}

An ideal $I$ of a $*$-ring $R$ is a {\it $*$-ideal} if $a^*\in I$ whenever $a\in I$.

\begin{prop} Let $I$ be a $*$-ideal of $R$. If
$R$ has $GC$ for elements, then $R/I$ has $GC$ for elements.\end{prop}
\begin{proof} Let $a+I$, $b+I\in R/I$. Applying $GC$ to $a,b\in
R$, there exists a central projection $h\in R$ such that
$ha\lesssim hb$ and $(1-h)b\lesssim (1-h)a$. Then passing to cosets,
$h+I$ is central projection in $R/I$ such that $(h+I)(a+I)\lesssim
(h+I)(b+I)$ and $[(1+I)-(h+I)](b+I)\lesssim [(1+I)-(h+I)](a+I)$. Hence $R/I$ has $GC$ for elements.
 \end{proof}

\begin{rem} The converse of above statement is not true. For, let $R=\mathbb{Z}_{10}$ with identity map as an involution and $I=\{0,2,4,6,8\}$. Then $R/I=\{0+I,1+I\}$
which has $GC$ for elements trivially. The poset $R$ with natural partial order is depicted in Figure $1$.
\begin{center}
\begin{tikzpicture}
\draw (0,1)--(-1.5,2)--(-2,1)--(0,0)--(-1,1)--(-0.5,2)--(0,1)--(0,0)--(1,1)--(0.5,2)--(0,1)--(1.5,2)--(2,1)--(0,0);
\draw [fill=white](0,1)
circle(.05);\draw [fill=white](-1.5,2)
circle(.05);\draw [fill=white](-2,1)
circle(.05);\draw [fill=white](0,0)
circle(.05);\draw [fill=white](-1,1)
circle(.05);\draw [fill=white](-0.5,2)
circle(.05);\draw [fill=white](1,1)
circle(.05);\draw [fill=white](0.5,2)
circle(.05);\draw [fill=white](1.5,2)
circle(.05);\draw [fill=white](2,1)
circle(.05);
\node [left]at (0,-.1){$0$};\node [above]at (-1.5,2){$7$};\node [left]at (-2,1){$2$};\node [left]at (-1,1){$6$};
\node [above]at (-0.5,2){$1$};\node [right]at (1,1){$4$};\node [above]at (0.5,2){$9$};\node [above]at (1.5,2){$3$};
\node [right]at (2,1){$8$};\node [left]at (0,0.9){$5$};\node [left]at (0.8,-1){Figure 1};

\end{tikzpicture}
\end{center}
Here $R$ does not have $GC$ for elements. On the contrary, if $R$ has $GC$ for elements, then by Theorem \ref{t4}, $R$ has $PC$ for elements. Let $a=2$ and $b=4$.
Then $aRb\neq 0$ and $2\nsim 4$, Since $22^*=4$ and $4^*4=6$ and $R$ being commutative there is no $x\in R$ such that $xx^*=4$ and $x^*x=6$. Hence $2$ and $4$ are not partially comparable in $R$, a contradiction.
\end{rem}

\end{document}